\documentclass[11pt]{amsart}
\usepackage{amssymb,latexsym,hyperref,amsmath}

\hypersetup{colorlinks=true,
linkcolor=blue,
filecolor=purple,      
urlcolor=cyan,
citecolor=magenta
}

\title[Gromov-Hausdorff stability of global attractors]{Gromov-Hausdorff stability of global attractors\\
of damped wave equations \\
under perturbations
of the domain}

\author{Ngoctu Bui and Jihoon Lee}

\address{Ngoctu Bui\hfill\break
Department of Mathematics, Chonnam National University, \hfill\break
Gwangju 61186, Korea}
\email{buingoctu22112000@gmail.com}

\address{Jihoon Lee\hfill\break
Department of Mathematics, Chonnam National University, \hfill\break
Gwangju 61186, Korea}
\email{jihoon@jnu.ac.kr}


\begin{document}
\maketitle
\numberwithin{equation}{section}
\newtheorem{theorem}{Theorem}[section]
\newtheorem{remark}{Remark}[section]
\newtheorem{definition}{Definition}[section]
\newtheorem{lemma}[theorem]{Lemma}
\newtheorem{corollary}[theorem]{Corollary}
\newtheorem{proposition}[theorem]{Proposition}
\newtheorem{example}[theorem]{Example}

\renewcommand{\footskip}{.5in}

\begin{center}
{\small \it Department of Mathematics, Chonnam National University \\
Gwangju, 61186, Korea}
\end{center}

\thispagestyle{empty}

\begin{abstract}
In this paper, we will make use of the Gromov-Hausdorff distance between compact metric spaces to establish the continuous dependence and the Gromov-Hausdorff stability of global attractors for damped wave equations under perturbations of the domain.
\end{abstract}

\section{ \textbf{INTRODUCTION}}

Let $\Omega$ be an open bounded domain in $\mathbb R^N$ $(N \le 3)$ with a smooth boundary, and $\varepsilon > 0$ a real parameter. Consider the following damped wave equation  (hyperbolic equation) with Dirchlet boundary conditions:
\begin{equation} \label{model}
\begin{cases}
\varepsilon u_{tt} + u_t - \Delta u = -f(u) ~& \mbox{in}\quad \Omega \times (0,\infty), \\
u=0 &\mbox{on} \quad \partial \Omega \times (0, \infty), \\
u(0,x) = u_0(x), \quad u_t(0,x)=u_1(x)& \mbox{in}  \quad \Omega.
\end{cases}
\end{equation}

Suppose $f$ belongs to $C^2(\mathbb{R}, \mathbb{R})$ and  satisfies the following condition: 
\begin{align*}
   \limsup_{|u|\rightarrow +\infty} {\frac{f(u)}{u}} > 0, \quad |f'(u)| \leq l
\end{align*}
and for $N\geq 2$, there exists a constant $C>0$ such that 
\begin{align*}
    |f''(u)|\leq C(|u|^\gamma +1) \mbox{ for $u \in \mathbb{R}$},
\end{align*}
where $ 0 \leq \gamma <\infty \mbox{ if } N=2$,  and $ 0\leq \gamma \leq1 \mbox{ if } N=3. $
We assume that $U_0 = (u_0, u_1)$ belongs to $H^1_0(\Omega) \times L^2(\Omega)$. Then we know that  the equation \eqref{model} is well-posed and has the global attractor in the product space $H^1_0(\Omega) \times L^2(\Omega)$ (see \cite{JK}).

Note that  \eqref{model} is a reaction diffusion equation with Dirchlet boundary condition if $\varepsilon = 0$, and the reaction diffusion
equation also has the global attractor in $H^1_0(\Omega)$ which can be naturally embedded into a compact set in the product space $H^1_0(\Omega) \times L^2(\Omega)$.

In the notable paper by J. K. Hale and G. Raugel  \cite{HR1}, they showed that the global attractors of the damped wave equation \eqref{model} are upper semicontinuous at the global attractor of the reaction diffusion equation (i.e., $\varepsilon=0$).
In a subsequent paper \cite{HR2}, they proved that if all of the equiliblium points of the reaction diffusion equation  are hyperbolic, then 
the global attractors of \eqref{model} are lower semicontinuous at the global attractor of the reaction diffusion equation.

Recently there are a lot of papers on the continuity and stability of global attractors (or inertial manifolds) of parabolic equations under perturbations of the domain and equation (e.g., see \cite{JK, HR1, HR2, LN, LT, LNP, LNP2, LNT, LP, LP2}). In particular, J. Lee et. al. \cite{LNT} considered the continuity and the stability of the dynamical systems on the global attractors induced by the reaction diffusion equations (i.e., $\varepsilon=0$) under perturbation of the domain.

In this paper, 
we study the long time behavior of solutions of hyperbolic equation \eqref{perturbed} obtained by perturbations of the domain in \eqref{model}. In particular,
we employ the Gromov-Hausdorff distance between compact metric spaces to establish the continuous dependence and the Gromov-Hausdorff stability of global attractors for damped wave equations under perturbations of the domain.

For this, we let $\mathcal P(\Omega)$ be the set of $C^2$ diffeomorphisms $h$ from $\Omega$ into $\mathbb R^N$ with $d_{C^2}(h, 1_{\Omega}) < 1$, where $1_{\Omega}$  denotes the identity map on $\Omega$. 
Suppose  $\mathcal P(\Omega)$  has the topology induced by the $C^2$ distance $d_{C^2}$. Then we see that the map $dist: \mathcal P(\Omega)  \to [0, 1)$  given by $dist(h)= d_{C^2}(h, 1_{\Omega})$ is continuous and surjective, and so $\mathcal P(\Omega)$ can be considered as the set of $\delta$-perturbations of $\Omega$ under the distance $d_{C^2}$, where $\delta \in [0, 1)$.


For any $0\leq \delta <1,$ let us consider the equation with domain perturbation in \eqref{model}:
\begin{equation} \label{perturbed-epsilon}
\begin{cases}
\varepsilon u_{tt} + u_t - \Delta u = -f(u)  ~& \mbox{in}\quad \Omega_\delta \times (0,\infty), \\
u=0 &\mbox{on} \quad \partial \Omega_\delta \times (0, \infty), \\
u(0,x) = u_{0,\delta}(x), \quad u_t(0,x)=u_{1,\delta}(x) & \mbox{in}  \quad \Omega_\delta,
\end{cases}
\end{equation}
where  $\Omega_{\delta} = h (\Omega)$ for some $h \in \mathcal P (\Omega)$ with $d_{C^2}(h,1_\Omega)=\delta$. Note that $\Omega_0 =\Omega$, $u_{0,0}=u_0,$ and $u_{1,0}=u_1$.

To simply the notation, we will consider with the case $\varepsilon = 1$ in the equation \eqref{perturbed-epsilon} as follows:
\begin{equation} \label{perturbed}
\begin{cases}
u_{tt} + u_t - \Delta u = -f(u)  ~& \mbox{in}\quad \Omega_\delta \times (0,\infty), \\
u=0 &\mbox{on} \quad \partial \Omega_\delta \times (0, \infty), \\
u(0,x) = u_{0,\delta}(x), \quad u_t(0,x)=u_{1,\delta}(x) & \mbox{in}  \quad \Omega_\delta, 
\end{cases}
\end{equation}
 Then we see that  the equation \eqref{perturbed} is also well-posed in various function spaces. Note that the asymptotic behavior of solutions of the equations \eqref{perturbed-epsilon} and \eqref{perturbed} are equivalent.

Let  $T_\delta (t)$, $t \geq0$, be the semidynamical system on the product space $H_0^1(\Omega_{\delta}) \times L^2(\Omega_{\delta})$   induced by the equation \eqref{perturbed}, that is,  
$$T_\delta (t): H_0^1(\Omega_{\delta}) \times L^2(\Omega_{\delta}) \to H_0^1(\Omega_{\delta}) \times L^2(\Omega_{\delta}), \quad T_\delta (t) ( u_{0,\delta}, u_{1,\delta})= (u(t), \partial_t u(t)),$$
where $ u(t)$ is the solution of  \eqref{perturbed}. Note that the operator $T_\delta (t)$ also yields a semidynamical system on the product space $(H^1_0(\Omega_\delta) \cap H^2(\Omega_\delta))\times H^1_0(\Omega_\delta)$.

In the following, we set
$$X^0_\delta :=H_0^1(\Omega_{\delta}) \times L^2(\Omega_{\delta}) \quad {\rm and} \quad
       X^1_\delta := (H^1_0(\Omega_\delta) \cap H^2(\Omega_\delta))\times H^1_0(\Omega_\delta),$$
and give the norm $\| \cdot \|_{X^s_{\delta}}$ on the product space $ X^s_{\delta}$ ($s=0$ or $1$) defined by
$$\|(u,v)\|_{X^s_{\delta}} = \left(\|u\|^2_{s+1} +\|v\|^2_s\right)^{1/2}$$
for $(u,v)\in X^s_{\delta}$. Here $\|u\|_0 = \|u\|_{L^2(\Omega_\delta)}, \|u\|_1 = \|A_\delta^{1/2}u\|_0$, $\|u\|_2 = \|A_\delta u\|_0$, and $A_\delta$ is  induced by the operator $-\Delta$ defined on $\Omega_\delta$.

We say that a subset $\mathcal A$ of $X^0_{\delta}$ is the global attractor for $T_\delta (t)$ if it is a compact invariant set which attracts all bounded sets in $X^0_{\delta}$.
Under the above hypotheses, we can see that  $T_\delta (t)$ admits the global attractor $\mathcal{A}_\delta $  in $X^0_{\delta}$ (also in $X^1_{\delta}$). In the case, we say that the equation \eqref{perturbed} has the global attractor  $\mathcal{A}_\delta $.  Note that $T_\delta (t)$ is injective on  $\mathcal{A}_\delta $. Hence we know that the restriction of the semidynamical system $T_\delta (t)$ to $\mathcal{A}_\delta $ gives rise to a dynamical system on $\mathcal A_\delta$. We recall that a dynamical system is a pair $(X, \phi)$ such that $X$ is a topological space and $\phi$ a continuous map $\phi: X \times \mathbb R \to X$ satisfying 
$$\phi(x, 0) =x \quad {\rm and}\quad \phi(\phi(x, s), t)=\phi(x, s+t), \quad \forall x \in X,  ~ s,t \in \mathbb R.$$

Let $\mathcal{M}$ be the collection of all compact metric spaces up to isometry.
For any $X, Y \in \mathcal{M}$ and $\varepsilon>0$, an $\varepsilon$-immersion from $X$ to $Y$ is a map $i: X \rightarrow Y$ such that
$$
\left|d_{X}(x, x^{\prime})-d_{Y}(i(x), i(x^{\prime})\right)|<\varepsilon, \quad \forall x, x^{\prime} \in X.
$$
We say that $i: X \rightarrow Y$ is an $\varepsilon$-isometry from $X$ to $Y$ if it is an $\varepsilon$-immersion and satisfies $B(i(X), \varepsilon)=Y$, where $B(i(X), \varepsilon)=\left\{y \in Y: d_{Y}(i(x), y)<\varepsilon\right.$ for some $\left.x \in X\right\}$. The Gromov Hausdorff distance $d_{G H}(X, Y)$ between $X$ and $Y$ is defined as the infimum of $\varepsilon>0$ for which there are $\varepsilon$-isometries $i: X \rightarrow Y$ and $j: Y \rightarrow X$.

Let $\mathcal{DS}$ denote the collection of all dynamical systems on elements of $\mathcal{M}$.
For any $\varepsilon>0$ and a topological space $X$, we denote by $\operatorname{Rep}_{X}(\varepsilon)$ the collection of all continuous maps $\alpha: X \times \mathbb{R} \rightarrow \mathbb{R}$ such that for each $x \in X, \alpha(x,\cdot): \mathbb{R} \rightarrow \mathbb{R}$ is a homeomorphism satisfying $\alpha(x, 0)=0$ for all $x \in X$ and $|\alpha(x, t)-t|<\varepsilon$ for any $t \in \mathbb{R}$. For any $(X, \phi)$ and $(Y, \psi)$ in  $\mathcal{DS}$, the Gromov Hausdorff distance $ D_{GH}(\phi, \psi) $ between $(X, \phi)$ and $(Y, \psi)$  is the infimum of $\varepsilon>0$ for which there are $\varepsilon$-isometries $i: X \rightarrow Y$ and $j: Y \rightarrow X$, and $\alpha \in \operatorname{Rep}_{X}(\varepsilon)$, $\beta\in \operatorname{Rep}_{Y}(\varepsilon)$ such that for any $x\in X$, $y\in Y$ and $t\in [0,1]$,
$$d(i(\phi(x, \alpha(x,t))),\psi(i(x),t))<\varepsilon \text{ and }\\d(j(\psi(y, \beta(y,t))),\phi(j(y),t))<\varepsilon$$
(for more details, see \cite{LM}).

\begin{definition}
    We say that the global attractor $\mathcal{A}_{\delta}$ of the equation $\eqref{perturbed}$ is Gromov-Hausdorff stable under perturbations of the domain if for any $\varepsilon>0$ there exists $\gamma>0$ such that if $d_{C^{2}}(g,h)<\gamma$, ($g \in \mathcal P(\Omega)$ with $d_{C^2}(g, 1_\Omega)=\tilde{\delta}$) then there are $\eta$-isometries $I: \mathcal{A}_{\delta} \rightarrow \mathcal{A}_{\tilde{\delta}}$ and $J: \mathcal{A}_{\tilde{\delta}} \rightarrow \mathcal{A}_{\delta}$, and $\alpha \in \operatorname{Rep}_{\mathcal{A}_{\delta}}(\varepsilon)$ and $\beta \in \operatorname{Rep}_{\mathcal{A}_{\tilde{\delta}}}(\varepsilon)$ such that for any $u_{0} \in \mathcal{A}_{\delta}$, $\tilde{u}_{0} \in \mathcal{A}_{\tilde{\delta}}$ and $t \in[0,1]$,
$$
\begin{aligned}
& \left\| I\left(T_{\delta}\left(u_{0}, \alpha\left(u_{0}, t\right)\right)\right)-T_{\tilde{\delta}}\left(I\left(u_{0}\right), t\right)\right) \|_{X_{\tilde\delta}^0}<\varepsilon, \text { and } \\
& \left\| J\left(T_{\tilde{\delta}}\left(\tilde{u}_{0}, \beta\left(\tilde{u}_{0}, t\right)\right)\right)-T_{\delta}\left(J\left(\tilde{u}_{0}\right), t\right)\right) \|_{X_\delta^0}<\varepsilon .
\end{aligned}
$$
\end{definition}

A subset $\mathcal{R}$ of a topological space $X$ is said to be residual if $\mathcal{R}$ contains a countable intersection of open dense subsets in $X$. We say that a function $f : X \rightarrow Y$ is residually continuous if
there exists a residual subset $\mathcal{R} $ of $X$ such that $f$ is continuous at every point of $\mathcal{R}$, where $Y$ is a topological space.

For any $h_0, h_n \in \mathcal P(\Omega)$, we consider the map $i_n: L^2(\Omega_{\delta_n})\rightarrow L^2(\Omega_{\delta_0})$
$$i_n(u) =u\circ(h_n\circ h_0^{-1}) \text { for } u \in L^2(\Omega_{\delta_n}),$$
and its inverse $i_n^{-1}: L^2(\Omega_{\delta_0})\rightarrow L^2(\Omega_{\delta_n})$
$$i_n^{-1}(v) =v\circ(h_0\circ h_n^{-1}) \text { for } v \in L^2(\Omega_{\delta_0}), $$
where $d_{C^2}(h_0, 1_\Omega)=\delta_0$ and $d_{C^2}(h_n, 1_\Omega) =\delta_n$.
Then we see that $i_n$ and $i_n^{-1}$ are bounded isomorphisms, and the restrictions of $i_n$ on the spaces $H^1_0(\Omega_{\delta_n})$ and $H^2({\Omega_{\delta_n}})$  are also bounded isomorphisms. Moreover 
the restrictions of $i_n^{-1}$ on the spaces $H^1_0(\Omega_{\delta_0})$ and $H^2({\Omega_{\delta_0}})$  are  bounded isomorphisms.

Now we define the map $I_n$ on the product space  $X^0_{\delta_n}= H_0^1(\Omega_{\delta_n}) \times L^2(\Omega_{\delta_n})$ given by 
\begin{equation}\label{I_n}
I_n: X^0_{\delta_n}\rightarrow X^0_{\delta_0}, \ \
I_n(U) =(i_n(u_0), i_n(u_1)) \text { for } U=(u_0, u_1) \in X^0_{\delta_n}
\end{equation}
and its inverse $I_n^{-1}$ on the product space $X^0_{\delta_0}= H_0^1(\Omega_{\delta_0}) \times L^2(\Omega_{\delta_0})$ given by
$$ I_n^{-1}: X^0_{\delta_0}\rightarrow X^0_{\delta_n}, \ \  I_n^{-1}(V) =(i_n^{-1}(v_0), i_n^{-1}(v_1)) \text { for } V=(v_0, v_1) \in X^0_{\delta_0}.$$
Then we see that $I_n$ and $I_n^{-1}$ are bounded isomorphisms, and the restriction of  $I_n$ on the product space $X^1_{\delta_n}$ and the restriction of  $I_n^{-1}$ on the product space $X^1_{\delta_0}$ are also bounded isomorphisms.

With all the notations in mind, we state our main results.

\begin{theorem}\label{thm1}
    The map $\mathcal{G}: \mathcal P(\Omega) \rightarrow\mathcal{M}$ given by $\mathcal{G}(h) = \mathcal{A}_{\delta}$ is residually continuous, where $d_{C^2}(h, 1_\Omega)=\delta$ and $\mathcal{A}_{\delta}$ is the global attractor of \eqref{perturbed}.
\end{theorem}

\begin{theorem}\label{thm2}
    The global attractor $\mathcal{A}_{\delta}$ induced by $\eqref{perturbed}$ is residually Gromov-Hausdorff stable under perturbations of the domain. More precisely, the map $\Phi: \mathcal P(\Omega)\to \mathcal{DS}$ given by $\Phi(h) = T_{\delta}$ is residually continuous, where $T_{\delta}$ is the dynamical system on $\mathcal{A}_\delta$, and $\delta =d_{C^2}(h,1_\Omega)$.
\end{theorem}

\section{{PROOF OF THEOREM \ref{thm1}}}

To prove the therem, we begin with  the following well-known lemma concerning the existence and uniqueness of the solutions to the given equation \eqref{perturbed}.

\begin{lemma}\label{lem wellposed}
    The perturbed equation \eqref{perturbed} has a unique weak solution $u(t)$ which satisfies  $(u(t), \partial_t u(t))\in C([0,\infty);X_\delta^0)$.
\end{lemma}

\begin{lemma}\label{theo21}
    Let $0\leq \delta\leq 1$. There exist a positive constant $C>0$ and, for any $r_0>0$, $r_1>0$, two positive constants $C^*(r_0)$ and $C^*(r_1)$ such that, for any solution $u_\delta(t)$ of equation \eqref{perturbed} with $\|u_{1,\delta}\|^2_i + \|u_{0,\delta}\|^2_{i+1} \leq r_i ^2, (i=0,1),$ we have the following estimate: for any $t\geq 0$,
    \begin{align*}
        \left\|\frac{\partial^2 u_\delta (t)}{\partial^2 t}\right\|^2_0 +\left\|\frac{\partial u_\delta (t)}{\partial t}\right\|^2_1 + \|u_\delta(t)\|^2_2 \leq C^*(r_0)+ C^*(r_1)e^{-C t}.
    \end{align*}
\end{lemma}

\begin{proof}
    See the proof of Theorem 2.5 in \cite{HR1}.
\end{proof}

Now we introduce a lemma which is crucial for the proof of Theorem \ref{thm1}.
\begin{lemma}\label{lem in}
    Let $\{h_n\}_{n\in \mathbb{N}}$ be a sequence in $\mathcal P(\Omega)$, and $\{V_{0,n}\}_{n\in \mathbb{N}}$ be a sequence in $X^0_{\delta_0}$  which is bounded in $X^1_{\delta_0}$ such that
    \begin{align*}
        h_n \rightarrow h_0 \mbox{ and } V_{0,n}\rightarrow V_0 \mbox{ as } n\rightarrow \infty.
    \end{align*}
    Then for any $T>0$, we have
    \begin{align*}
        \|I_n(T_{\delta_n}(I_n^{-1}(V_{0,n}),t))-T_{\delta_0}(V_0, t)\|_{X^0_{\delta_0}}\rightarrow 0 \mbox{ for any } t\in [0, T],
    \end{align*}
    where $\delta_n = d_{C^2}(h_n, 1_\Omega)$ for $n = 0,1,2,...$. 
\end{lemma}

\begin{proof}
    Take $T>0$, and let $U_{0,n} = I_n^{-1}(V_{0,n})=(u_{0,n},u_{1,n})$. Let  $u_n(t)$ be the weak solution of the equation
    \begin{align}\label{eq:3.1}
        \left\{\begin{array}{ll}
u_{tt} + u_t - \Delta u = -f(u) ~& \mbox{in}\quad \Omega_{\delta_n} \times (0,\infty), \\
u=0 &\mbox{on} \quad \partial \Omega_{\delta_n} \times (0, \infty), \\
u(0,x) = u_{0,n}(x), \quad u_t(0,x)=u_{1,n}(x),& \mbox{in}  \quad \Omega_{\delta_n},
\end{array}\right.\end{align}
and $v(t)$  be the weak solution of the equation
\begin{align}\label{eq:3.2}
    \left\{\begin{array}{ll}
v_{tt} + v_t - \Delta v = -f(v)~& \mbox{in}\quad \Omega_{\delta_0} \times (0,\infty), \\
v=0 &\mbox{on} \quad \partial \Omega_{\delta_0} \times (0, \infty), \\
v(0,x) = v_0(x), \quad v_t(0,x)=v_1(x),& \mbox{in}  \quad \Omega_{\delta_0},
\end{array}\right.\end{align}
Let $U_n(t) = (u_n(t), \partial_t u_n(t))$, $V(t) = (v(t), \partial_t v(t))$, $V_n = I_n(U_n)$ and $Z_n = V_n - V$. We need to show that $\|Z_n (t)\|_{X^0_{\delta_0}}\rightarrow 0 \mbox{ for all } t\in [0, T].$ 

Since $V_{0,n} $ is bounded in $X^1_{\delta_0}$ for $n=0,1,2,...$, and $I_n^{-1} $ is a bounded isomorphism from $X_{\delta_0}^i$ to $X_{\delta_n}^i$ for $i\in \{0,1\}$, we have that the initial values $U_{0,n}$ is also bounded in $X_{\delta_n}^1$. Hence, by Lemma \ref{theo21}, we obtain the boundedness of $U_n(t)$ in $X_{\delta_n}^1$. Similarly we see that $V_n(t)$ belongs to the space $X_{\delta_0}^1$. Hence we know that $Z_n \in X_{\delta_0}^1$, and so $i_n^{-1}(\partial_t z_n) \in L^2(0,T;H_0^1(\Omega_{\delta_n}))$. We take the inner product with $i_n^{-1}(\partial_t z_n)$ on the both sides of equation \eqref{eq:3.1} to deduce that
\begin{align}\label{eq:3.3}
\int_{\Omega_{\delta_n}}\partial_{tt}u_n (i_n^{-1}(\partial_t z_n)) dx &+ \int_{\Omega_{\delta_n}}\partial_tu_n (i_n^{-1}(\partial_t z_n)) dx \notag\\
&+\int_{\Omega_{\delta_n}}\nabla u_n \cdot \nabla (i_n^{-1}(\partial_tz_n))dx 
=-\int_{\Omega_{\delta_n}}f(u_n) i_n^{-1}(\partial_t z_n) dx.
\end{align}
Let $y = h_0 \circ h_n^{-1}(x)$ for $x \in \Omega_{\delta_n}$. By the change of variable and \eqref{eq:3.3} (using the notation $x$ instead of $y$ for convenience), we get  that
\begin{align} \label{eq:3.4}
&\int_{\Omega_{\delta_0}}\partial_{tt}v_n + \partial_t z_n|\det H_n|dx+ \int_{\Omega_{\delta_0}}\partial_{t}v_n  \partial_t z_n|\det H_n|dx\notag \\
&+\int_{\Omega_{\delta_0}}\overline{H}_n\nabla v_n \cdot \overline{H}_n\nabla \partial_t z_n |\det H_n| dx =-\int_{\Omega_{\delta_0}} f(v_n) \partial_t z_n|\det H_n| dx,
\end{align}
where
\begin{equation*}\label{eq:3.5}
H_n(x)=D(h_n\circ h_0^{-1}(x))=\left(\dfrac{\partial (h_n\circ h_0^{-1}(x))_i}{\partial x_j}\right)_{N\times N}, \; x\in \Omega_{\delta_0},
\end{equation*}
and $\overline{H}_n=(H_n^{-1})^{\mathsf{T}}$ is the transpose of the inverse $H_n^{-1}$.
\vskip 5 mm
We take the inner product with $\partial_t z_n$ on both sides of the equation \eqref{eq:3.2} to deduce that
\begin{equation}\label{eq:3.6}
\int_{\Omega_{\delta_0}}v_{tt}\partial_t z_n dx+\int_{\Omega_{\delta_0}}v_t \partial_t z_n dx+ \int_{\Omega_{\delta_0}}\nabla v\cdot \nabla \partial_t z_ndx=-\int_{\Omega_{\delta_0}}f(v) \partial_t z_ndx.\end{equation}
By subtracting \eqref{eq:3.6} from \eqref{eq:3.4}, we get
\begin{align}\label{eq:3.7}
    &\int_{\Omega_{\delta_0}}\partial_{tt}z_n \partial_tz_n |\det H_n| dx+\int_{\Omega_{\delta_0}}\partial_{t}z_n \partial_tz_n |\det H_n| dx\\
&\;\;\;+\int_{\Omega_{\delta_0}}(|\det H_n|-1)(v_t+v_{tt}) \partial_tz_n dx +\int_{\Omega_{\delta_0}}\overline{H}_n\nabla z_n \cdot \overline{H}_n\nabla \partial_t z_n |\det H_n| dx\notag\\
&=\int_{\Omega_{\delta_0}}\left(I-\overline{H}_n\right)\nabla v \cdot \overline{H}_n\nabla \partial_tz_n |\det H_n| dx\notag\\
&\;\;\;+\int_{\Omega_{\delta_0}}\nabla v \cdot (I-\overline{H}_n)\nabla \partial_tz_n |\det H_n|dx+\int_{\Omega_{\delta_0}}(1-|\det H_n|)\nabla v\cdot \nabla \partial_tz_ndx\notag\\
&\;\;\;+\int_{\Omega_{\delta_0}}(1-|\det H_n|)f(v) \partial_tz_ndx-\int_{\Omega_{\delta_0}} (f(v_n)-f(v)) \partial_tz_n|\det H_n| dx\notag.
\end{align}

Now we estimate each term. For the left hand side of \eqref{eq:3.7}, the first term is
\begin{align*}
    \int_{\Omega_{\delta_0}}\partial_{tt}z_n \partial_tz_n |\det H_n| dx = \frac{1}{2} \frac{d}{dt}\int_{\Omega_{\delta_0}}\left|\partial_t z_n\right|^2 |\det H_n| dx,
\end{align*}
the second term is
\begin{align*}
    \int_{\Omega_{\delta_0}}\partial_t z_n \partial_tz_n |\det H_n| dx = \int_{\Omega_{\delta_0}}\left|\partial_t z_n\right|^2 |\det H_n| dx,
\end{align*}
the third term is
\begin{align*}
    - \int_{\Omega_{\delta_0}}(|\det H_n|-1)(v_t+v_{tt}) \partial_tz_n dx
    &\leq \||\det H_n| -1\|_{\infty} \int_{\Omega_{\delta_0}}\left(- \nabla v\cdot \nabla z_n - f(v) z_n \right)dx\\
    &\leq \frac{1}{2}\||\det H_n| -1\|_{\infty} K(t),
\end{align*}
where $K(t) = \|v\|_1^2 + \|\partial_t z_n\|^2_1 +\|f(v)\|^2_0 +\|\partial_t z_n\|^2_0$, 
and the fourth term is
\begin{align*}
    \int_{\Omega_{\delta_0}}\overline{H}_n\nabla z_n \cdot \overline{H}_n\nabla \partial_t z_n |\det H_n| dx = \frac{1}{2} \frac{d}{dt}\int_{\Omega_{\delta_0}}|\overline{H}_n\nabla z_n|^2|\det H_n|dx.
\end{align*}

For the right hand side of \eqref{eq:3.7}, we have
\begin{align*}
    &\int_{\Omega_{\delta_0}}\left(I-\overline{H}_n\right)\nabla v \cdot \overline{H}_n\nabla \partial_tz_n |\det H_n| dx \\&\;\;\;+\int_{\Omega_{\delta_0}}\nabla v \cdot (I-\overline{H}_n)\nabla  \partial_tz_n |\det H_n|dx+\int_{\Omega_{\delta_0}}(1-|\det H_n|)\nabla v\cdot \nabla \partial_tz_ndx \\
    &\leq \dfrac{1}{2}\Big\{\|I-\overline{H}_n\|_{\infty}\|\det H_n\|_{\infty}\left(\|\overline{H}_n\|_{\infty}+1\right) +\||\det H_n |-1\|_{\infty}\Big\}\left(\|v\|_1^2+\|\partial_t z_n\|_1^2\right),
\end{align*}
the next term can be estimated as
\begin{align*}
    \int_{\Omega_{\delta_0}}(1-|\det H_n|)f(v) \partial_t z_ndx & \leq \||\det H_n |-1\|_{\infty}\|f(v)\|_0\|\partial_tz_n\|_0\\
&\leq \||\det H_n |-1\|_{\infty}\left(\dfrac{1}{2}\|f(v)\|^2_0+\dfrac{1}{2}\|\partial_tz_n\|^2_0\right),
\end{align*}
and the last term can be estimated by
\begin{align*}
    &-\int_{\Omega_{\delta_0}} (f(v_n)-f(v)) \partial_t z_n|\det H_n| dx \leq l \int_{\Omega_{\delta_0}}|z_n \partial_t z_n||\det H_n|dx \\
    &\leq \frac{l}{2}\|z_n \sqrt{|\det H_n|}\|^2_0 +\frac{l}{2}\|\partial_t z_n \sqrt{|\det H_n|}\|^2_0 \leq \frac{\ell}{2}(\|\nabla_{\delta_0} z_n \sqrt{|\det H_n|}\|^2_0 +\|\partial_t z_n \sqrt{|\det H_n|}\|^2_0)\\
    &\leq \ell(\|\overline{H}_n \nabla_{\delta_0}z_n \sqrt{|\det H_n|}\|^2_0+\|\partial_t z_n \sqrt{|\det H_n|}\|^2_0)+\ell \|I - \overline{H}_n\|^2_\infty\|\det H_n\|_{\infty}\|\nabla_{\delta_0}z_n \|_0^2,
\end{align*}
where $\ell = \max\left\{\frac{l}{\lambda_1}, l\right\}$, and $\lambda_1$ is the first eigenvalue of the operator $A_0$.

Consequently we derive
\begin{align*}
    &\frac{1}{2} \frac{d}{dt}\int_{\Omega_{\delta_0}}\left|\partial_t z_n\right|^2 |\det H_n| dx + \frac{1}{2} \frac{d}{dt}\int_{\Omega_{\delta_0}}|\overline{H}_n\nabla z_n|^2|\det H_n|dx \\
    &\leq \||\det H_n| -1\|_{\infty} K(t) \\
    &\;\;\;+\dfrac{1}{2}\Big\{\|I-\overline{H}_n\|_{\infty}\|\det H_n\|_{\infty}\left(\|\overline{H}_n\|_{\infty}+1\right) +\||\det H_n |-1\|_{\infty}\Big\}\left(\|v\|_1^2+\|\partial_t z_n\|_1^2\right)\\
    &\;\;\;+ \ell(\|\overline{H}_n \nabla z_n \sqrt{|\det H_n|}\|^2_0+\|\partial_t z_n \sqrt{|\det H_n|}\|^2_0)\\ 
    &\;\;\;+\ell\|I - \overline{H}_n\|^2_\infty\|\det H_n\|_{\infty}\|\nabla z_n \|_0^2, 
\end{align*}
which is equivalent to
\begin{align}\label{eq:3.11}
    &\frac{d}{dt}\left(e^{-2\ell t}\int_{\Omega_{\delta_0}}\left(\left|\partial_t z_n\right|^2 +|\overline{H}_n\nabla_{\delta_0} z_n|^2\right)|\det H_n|dx\right)\\
    &\leq 2 e^{-2\ell t} \||\det H_n| -1\|_{\infty} K(t) \notag\\
    & \;\;\;+ e^{-2\ell t} \Big\{\|I-\overline{H}_n\|_{\infty}\|\det H_n\|_{\infty}\left(\|\overline{H}_n\|_{\infty}+1\right)\big\}\left(\|v\|_1^2+\|\partial_t z_n\|_1^2\right) \notag \\
    &\;\;\; + e^{-2\ell t}||\det H_n |-1\|_{\infty}\left(\|v\|_1^2+\|\partial_t z_n\|_1^2\right)\notag\\
    &\;\;\;+ e^{-2\ell t} \ell\|I - \overline{H}_n\|^2_\infty\|\det H_n\|_{\infty}\|\nabla_{\delta_0}z_n \|_0^2\notag.
\end{align}
By integrating \eqref{eq:3.11} from $0$ to $t\in [0,T]$, we derive that
\begin{align}\label{eq:3.12}
& \int_{\Omega_{\delta_0}}\left(\left|\partial_t z_n\right|^2 +|\overline{H}_n\nabla_{\delta_0} z_n|^2\right)|\det H_n|dx\\
&\leq e^{2\ell t}\int_{\Omega_{\delta_0}}\left(\left|\partial_t z_n(0)\right|^2 +|\overline{H}_n\nabla_{\delta_0} z_n(0)|^2\right)|\det H_n|dx \notag \\
&\;\;\ + 2e^{2\ell t}\||\det H_n |-1\|_{\infty}\int_0^T K(t)dt \notag \\
& \;\;\;+ e^{2\ell t} \ell\|I - \overline{H}_n\|^2_\infty \|\det H_n\|_\infty\|z_n \|^2_{L^{\infty}(0, T, H^1_0(\Omega_{\delta_0}))}\notag\\
&\;\;\;+e^{2\ell t}\Big\{\|I-\overline{H}_n\|_{\infty}\|\det H_n\|_{\infty}\left(\|\overline{H}_n\|_{\infty}+1\right)+\||\det H_n |-1\|_{\infty}\Big\} \notag\\
&\;\;\;\times  \int_0^T\left(\|v(t)\|_1^2+\|\partial_t z_n(t)\|_1^2\right)dt\notag.
\end{align}

We now choose a constant $C$ independent of $n$ so that 
\begin{align*}
\max \big\{ \int_0^TK(t)dt,~\int_0^T\left(\|v(t)\|_1^2+\|\partial_t z_n(t)\|^2_1\right)dt,~ \|z_n \|^2_{L^{\infty}(0, T, H^1_0(\Omega_{\delta_0}))}\big\} \le C.
\end{align*}
Since $h_n \to h_0$ in $\mathcal{P}(\Omega)$ and $\det$ is a continuous operator, we can see that  $\|\det H_n\|_{\infty}$  and  $\|\overline{H}_n\|_{\infty} $ are uniformly bounded with respect to $n$. Moreover, note that 
$$
\inf\limits_{x \in \Omega_{h_0}}|\det H_n(x)| \ge \frac12 ~\mbox{and}~\inf\limits_{x \in \Omega_{h_0}} |\overline{H}_n(x)| \ge \frac12
$$ for sufficiently large $n \in \mathbb N$, and 
\begin{equation*}
\||\det H_n |-1\|_{\infty}\to 0\mbox{ and }   \|I-\overline{H}_n\|_{\infty}\to 0 \mbox{ as } n\to \infty.
\end{equation*}
Therefore,  from \eqref{eq:3.12} we conclude that
\begin{align*}
    &\frac{1}{4}\int_{\Omega_{\delta_0}}\left(\left|\partial_t z_n\right|^2 +|\nabla z_n|^2\right)dx\\
    &\leq \inf_{x\in \Omega_{\delta_0}}|\det H_n(x)|
    \int_{\Omega_{\delta_0}}\left(\left|\partial_t z_n\right|^2 +|\overline{H}_n(x)\nabla z_n|^2\right)dx\\
    &\leq e^{2\ell T} \|\det H_n\|_{\infty}\int_{\Omega_{\delta_0}}\left(\left|\partial_t z_n(0)\right|^2 +\|\overline{H}_n\|^2_\infty|\nabla z_n(0)|^2\right)dx \to 0 \mbox{ as } n\to \infty,
\end{align*}
which completes the proof of the lemma.
\end{proof}

To prove our main results, we need the following lemma (for the proof, see \cite{LM,LNT}).
\begin{lemma}\label{lemtopo1}
    Let $\Lambda$ be a topological space and $\{T_\lambda(t)\}_{\lambda \in \Lambda}$ $(t \geq 0)$  a family of semidynamical systems on metric spaces $\{X_\lambda\}_{\lambda \in \Lambda}$. Suppose that for each $\lambda \in \Lambda$,

    (1) $T_\lambda(t)$ has the global attractor $\mathcal{A}_\lambda$, and

    (2) there is a bounded neighborhood $D_\lambda$ of $\mathcal{A}_\lambda$ such that for any $t \geq0 $, $\lambda_0 \in \Lambda$, and 
    $\varepsilon>0$, there exists a neighborhood $\mathcal{W}$ of $\lambda_0$ such that for any $\lambda \in \mathcal{W}$, there is an $\varepsilon$-immersion $i: D_\lambda \rightarrow X_{\lambda_0}$ satisfying
   $$ i_\lambda\left(D_\lambda\right) \subset B\left(D_{\lambda_{0}}, \varepsilon\right) ~~  {\rm and } ~~ d(i_\lambda\left(T_{\lambda}\left(x, t\right)\right), T_{{\lambda_0}}\left(i_{\lambda}\left(x\right), t\right ))<\varepsilon ~~\forall x \in \mathcal{A}_{\lambda}.$$
    Then the map $\mathcal{G}: \Lambda \rightarrow\mathcal{M}$ given by $\mathcal{G}(\lambda) = \mathcal{A_\lambda}$ is residually continuous.
\end{lemma}
\begin{proof}[Proof of Theorem $\ref{thm1}$]
    We prove the theorem by showing that the assumptions of Lemma $\ref{lemtopo1}$ hold. For each $h\in \mathcal P(\Omega)$ with $d_{C^2}(h,1_\Omega)=\delta$, let $\Lambda = \mathcal P(\Omega)$ and $D_\delta$ be an absorbing ball in $X^1_\delta$, which is relatively compact in $X^0_\delta$. Then it is enough to show that for any $T > 0, \varepsilon > 0$ and $h_0 \in  \mathcal P(\Omega)$ with $d_{C^2}(h_0,1_\Omega)=\delta_0$ , there exists a neighborhood $\mathcal{W}$ of $h_{0}$ such that for any $h \in \mathcal{W}$ with $d_{C^2}(h,1_\Omega)=\delta$, there is an $\varepsilon$-immersion $I_\delta: D_{\delta} \rightarrow X^0_{\delta_0}$ satisfying
$$
 I_\delta \left(D_{\delta}\right) \subset B_{X^1_{\delta_0}}(D_{\delta_{0}}, \varepsilon) \text { and } 
\left\|I_\delta \left(T_{\delta}\left(U_{0}, t\right)\right)-T_{\delta_{0}}\left(I_\delta\left(U_{0}\right), t\right)\right\|_{X^0_{\delta_{0}}}<\varepsilon,
$$
where $U_{0} \in \mathcal{A}_{\delta}$ and $t \in[0, T]$.

Suppose not. Then there are $T>0, \varepsilon>0$ and $h_0 \in \mathcal P(\Omega)$ such that for any sequence $\left\{h_{n}\right\}_{n \in \mathbb{N}}$ converging to $h_{0}$, for any $\varepsilon$-immersion $I_{n}: D_{\delta_{n}} \rightarrow X^0_{\delta_{0}}$ with  $I_{n}\left(D_{\delta_{n}}\right) \subset B_{X^1_{\delta_0}}\left(D_{\delta_{0}}, \varepsilon\right)$, there exist $U_{0,n} \in \mathcal{A}_{\delta_{n}}$ and $t_{n} \in[0, T]$ such that
\begin{equation}\label{inequality}
\left\|I_{n}\left(T_{\delta_n}\left(U_{0,n}, t_{n}\right)\right)-T_{\delta_{0}}\left(I_{n}\left(U_{0,n}\right), t_{n}\right)\right\|_{X^0_{\delta_{0}}}>\varepsilon,
\end{equation}
where $\delta_n =d_{C^2}(h_n, 1_\Omega)$.

Define a map  $I_{n}: D_{\delta_{n}} \rightarrow X^0_{\delta_{0}}$ given by
$$
I_n(U) =U \circ (h_n \circ h_0^{-1}) \text{ for } U \in D_{\delta_n}^0.
$$
Then we see that $I_n$ is an $\varepsilon$-immersion for sufficiently large $n$. 
By our assumption, there are $U_{0,n} \in \mathcal{A}_{\delta_{n}}$ and $t_{n} \in[0, T]$ which satisfy the inequality \eqref{inequality}. Let $V_{0,n}=I_{n}\left(U_{0,n}\right)$ for each $n \in \mathbb{N}$. Then we see that $\left\{V_{0,n}\right\}$ is uniformly bounded in $X^1_{\delta_{0}}$. By the compact embedding theorem in \cite{Rb}, $\left\{V_{0,n}\right\}$ has a convergent subsequence in $X^0_{\delta_{0}}$, say
$$
V_{0,n} \rightarrow V_{0} \in X^0_{\delta_{0}} \text { as } n \rightarrow \infty.$$
Moreover, $\{V_{0,n}\}$ has a weakly convergent subsequence in $X^1_{\delta_0}$. By the uniqueness of weak limit, we get $V_0\in X^1_{\delta_0}.$
Consequently we have
$$
\left\|I_{n}\left(T_{\delta_{n}}\left(I_{n}^{-1}\left(V_{0,n}\right), t_{n}\right)\right)-T_{\delta_{0}}\left(V_{0}, t_{n}\right)\right\|_{X^0_{\delta_{0}}} \leq \frac{\varepsilon}{2}
$$
for  sufficiently large $n$. This contradicts to  Lemma $\ref{lem in}$, and so completes the proof of Theorem $\ref{thm1}$.
\end{proof}

\section{{PROOF OF THEOREM \ref{thm2}}}

In this section, we prove that residually the dynamical system on the global attractor $\mathcal A_\delta$ of the equation \eqref{perturbed} varies continuously under perturbations of the domain. To derive the result, we need the following lemma which is crucial for the proof of Theorem \ref{thm2}.

\begin{lemma}\label{continuity on h}
    For any $h_0\in \mathcal P(\Omega)$, there is a neighborhood $\mathcal{U}$ of $h_0$ such that the family of dynamical systems $\{T_\delta: \delta=d_{C^2}(h,1_\Omega)\}_{h \in \mathcal U} $ is equicontinuous on  $\mathcal{A}_\delta \times [0, 1]$.
\end{lemma}

\begin{proof} 
It is enough to show that for any $\varepsilon >0$, there is $\gamma >0$ such that  for any ${h\in \mathcal{U}}$, if $\|U_0 - \tilde{U}_0\|_{X^0_{\delta}} + |t-s| < \gamma$, then
    $$\| T_{\delta}(U_0, t) - T_{\delta}(\tilde{U}_0, s) \|_{X^0_{\delta}} < \varepsilon,$$
 for all $U_0, \tilde{U}_0 \in \mathcal{A}_\delta$   and  $t, s\in [0,1]$.

Step 1.
We first claim that for any $\varepsilon>0$, there is $\gamma_{1}>0$ such that for a small $\delta>0$,  if  $\| U_{0}-$ $\tilde{U}_{0} \|_{X^0_{\delta}}<\gamma_{1}$ for $U_0, \tilde{U}_0 \in X_\delta^0$, then
\begin{equation*}
\|U(t)-\tilde{U}(t)\|_{X^0_{\delta}}<\frac{\varepsilon}{2}, ~~\forall t \in[0,1],
\end{equation*}
where $U(t)=T_{{\delta}}\left(U_{0}, t\right)$ and $\tilde{U}(t)=T_{{\delta}}(\tilde{U}_{0}, t)$. Let $Z(t)=U(t)-\tilde{U}(t)$. Then $Z(t) = (z(t), \partial_t z(t))$ for any  $t \in \mathbb{R}^{+}$, where $z(t)$ is the weak solution of 
\begin{align} \label{diff}
\begin{cases}\partial_{tt} z +\partial_t z-\Delta z+f(u)-f(\tilde{u})=0 & \text { in } \Omega_{{\delta}} \times(0, \infty), \\ z=0 & \text { on } \partial \Omega_{{\delta}} \times(0, \infty), \\ 
z(0,x) = u(0,x)-\tilde{u}(0,x), z_t(0,x) = u_t(0,x)-\tilde{u}_t(0,x)& \text { in } \Omega_{{\delta}}.\end{cases}
\end{align}
Since $U(t), \tilde{U}(t) \in X^1_{\delta}$, we get  $Z(t)\in X^1_{\delta}$. By taking the inner product with $\partial_t z$ on the equation \eqref{diff}, we have
\begin{align}
    (\partial_{tt}z, \partial_t z) + (\partial_{t}z, \partial_t z) +(Az, \partial_t z) = -(f(u)-f(\tilde{u}), \partial_t z),\notag
\end{align}
which is equivalent to
\begin{align*}    
    \frac{1}{2} \frac{d}{dt}\|\partial_t z\|^2_0 + \|\partial_t z\|^2_0+ \frac{1}{2}\frac{d}{dt}\| z\|^2_1\leq &\frac{1}{2}\|f(u)-f(\tilde{u})\|^2_0+\frac{1}{2}\|\partial_t z\|^2_0.
\end{align*}
Using the conditions on $f$ and the inequality $\lambda_1 \|z\|^2_0 \leq \|z\|^2_1$, we have
\begin{align*}
    \frac{1}{2} \frac{d}{dt}\|\partial_t z\|^2_0 + \frac{1}{2}\frac{d}{dt}\| z\|^2_1 \leq &\frac{l}{2\lambda_1}\|z\|^2_1+\frac{1}{2}\|\partial_t z\|^2_0.
\end{align*}
Hence,  we get 
\begin{align}    \label{C}
    \frac{1}{2}\frac{d}{dt}\|Z(t)\|^2_{X^0_{\delta}}\leq \left(\frac{l}{2\lambda_1}+\frac{1}{2}\right)(\| z\|^2_1 +\|\partial_t z\|^2_0) 
   : = C\|Z\|^2_{X^0_{\delta}}.
\end{align}
Take $\gamma_1 >0$ satsifying $\gamma_1 < \frac{\varepsilon}{2e^C}$, where $C$ is a constant satisfying  \eqref{C}.
By the Gronwall's inequality, if $\|Z(0)\|_{X^0_{\delta}} < \gamma_1$, we have
\begin{align}\label{Lipschitz}
    \|Z\|^2_{X^0_{\delta}}\leq \|Z(0)\|^2_{X^0_{\delta}} e^{2Ct}<\frac{\varepsilon}{2}.
\end{align}

Step 2.
Next we show that for any $\varepsilon>0$ and $h_{0} \in \mathcal P(\Omega)$ with $d_{C^2}(h_0,1_{\Omega})=\delta_0$, there exist a neighborhood $\mathcal{U}$ of $h_{0}$ and a constant $\gamma_{2}>0$ such that for any $h \in \mathcal{U}$ with $d_{C^2}(h,1_{\Omega})=\delta$,
\begin{equation*}
\left\|T_{\delta}\left(U_{0}, s\right)-U_{0}\right\|_{X^0_{{\delta}}}<\frac{\varepsilon}{2}, \quad \forall\left(U_{0}, s\right) \in \mathcal{A}_{{\delta}} \times\left[0, \gamma_{2}\right] . 
\end{equation*}
For any $\varepsilon>0$, as in the proof of Theorem $\ref{thm1}$, we can take a neighborhood $\mathcal{U}$ of $h_{0}$ in $\mathcal P(\Omega)$ such that for any $h \in \mathcal{U}$, there is an $\frac{\varepsilon}{6}$-immersion $I_{{\delta}}: D_{{\delta}} \rightarrow {X^0_{{\delta}_{0}}}$ satisfying
\begin{align} \label{conti}
&I_\delta\left(D_{{\delta}}\right) \subset B\left(D_{{\delta}_{0}}, \frac{\varepsilon}{6}\right), \text { and } \notag \\
&\left\|I_{{\delta}}\left(T_{{\delta}}\left(U_{0}, s\right)\right)-T_{{\delta}_{0}}\left(I_{{\delta}}\left(U_{0}\right), s\right)\right\|_{X^0_{{\delta}_{0}}}<\frac{\varepsilon}{6}, \quad \forall(h, s) \in \mathcal{U}_2 \times[0,1],
\end{align}
where $d_{C^2}(h,1_{\Omega}) = \delta$, and $D_\delta$ is an absorbing neighborhood of $\mathcal A_\delta$.
Let $B_{{\delta}_0} = B(D_{{\delta}_0}, \frac{\varepsilon}{6})$ in $X^1_{{\delta}_0}$. Since $B_{{\delta}_{0}}$ is relatively compact in $X^0_{{\delta}_{0}}$ and $I_{{\delta}}\left(U_{0}\right) \in B_{{\delta}_{0}}$ for any $U_{0} \in \mathcal{A}_{{\delta}}$ (by definition of the immersion $I_{\delta}$), there exists $\gamma_{2}>0$ such that
\begin{equation*}
\left\|T_{{\delta}_{0}}\left(I_{{\delta}}\left(U_{0}\right), s\right)-I_{{\delta}}\left(U_{0}\right)\right\|_{X^0_{{\delta}_{0}}}<\frac{\varepsilon}{6}, \quad \forall\left(U_{0}, s\right) \in \mathcal{A}_{{\delta}} \times\left[0, \gamma_{2}\right] . 
\end{equation*}

Suppose not. Then for each $n \in \mathbb{N}$, there are $s_{n} \in[0,1 / n]$ and $U_{0,n} \in \mathcal{A}_{{\delta}}$ satisfying
$$
\left\|T_{{\delta}_{0}}\left(I_{{\delta}}\left(U_{0,n}\right), s_{n}\right)-I_{{\delta}}\left(U_{0,n}\right)\right\|_{X^0_{{\delta}_{0}}}\geq \frac{\varepsilon}{6} .
$$
Since $B_{{\delta}_{0}}$ is relatively compact in $X^0_{{\delta}_{0}}$, we may assume that $I_{{\delta}}\left(U_{0,n}\right)$ converges to a point, say $V_{0} \in X^0_{{\delta}_{0}}$, as $n$ tends to infinity. By $\eqref{Lipschitz}$, we have
$$
\left\|T_{{\delta}_{0}}\left(V_{0}, s_{n}\right)-T_{{\delta}_{0}}\left(I_{{\delta}}\left(U_{0,n}\right), s_{n}\right)\right\|_{X^0_{{\delta}_{0}}} <e^{C}\left\|V_{0}-I_{{\delta}}\left(U_{0,n}\right)\right\|_{X^0_{{\delta}_{0}}}.
$$
Hence we can choose $N_{1}>0$ such that for any $n>N_{1}$,
$$
\max \left\{\left\|T_{{\delta}_{0}}\left(V_{0}, s_{n}\right)-T_{{\delta}_{0}}\left(I_{{\delta}}\left(U_{0,n}\right), s_{n}\right)\right\|_{X^0_{{\delta}_{0}}},\left\|I_{{\delta}}\left(U_{0,n}\right)-V_{0}\right\|_{X^0_{{\delta}_{0}}}\right\}<\frac{\varepsilon}{18}.
$$
Note that $s_{n} \rightarrow 0$ as $n \rightarrow \infty$. By the continuity of $T_{{\delta}_{0}}\left(V_{0},.\right)$, there exits $N_{2}>0$ such that
$$
\left\|T_{{\delta}_{0}}\left(V_{0}, s_{n}\right)-V_{0}\right\|_{X^0_{{\delta}_{0}}}<\frac{\varepsilon}{18}, \quad \forall n>N_{2}.
$$
Hence for any $n>\max \left\{N_{1}, N_{2}\right\}$, we get
\begin{align*}
0=\left\|T_{{\delta}_{0}}\left(V_{0}, 0\right)-V_{0}\right\|_{X^0_{{\delta}_{0}}} \geq & \left\|T_{{\delta}_{0}}\left(I_{{\delta}}\left(U_{0,n}\right), s_{n}\right)-I_{{\delta}}\left(U_{0,n}\right)\right\|_{X^0_{{\delta}_{0}}}\\
& -\left\|T_{{\delta}_{0}}\left(I_{{\delta}}\left(U_{0,n}\right), s_{n}\right)-T_{{\delta}_{0}}\left(V_{0}, s_{n}\right)\right\|_{X^0_{{\delta}_{0}}}\\
& -\left\|T_{{\delta}_{0}}\left(V_{0}, s_{n}\right)-T_{{\delta}_{0}}\left(V_{0}, 0\right)\right\|_{X^0_{{\delta}_{0}}} \\
& -\left\|I_{{\delta}}\left(U_{n, 0}\right)-V_{0}\right\|_{X^0_{{\delta}_{0}}}>0 .
\end{align*}
The contradiction proves \eqref{conti}.
Consequently we deduce that for any $h \in \mathcal{U}, U_{0} \in \mathcal{A}_{\delta}$ and $s \in\left[0, \gamma_{2}\right]$,
\begin{align*}
\left\|T_{{\delta}}\left(U_{0}, s\right)-U_{0}\right\|_{X^0_{{\delta}}} \leq & \frac{\varepsilon}{6}+\left\|I_{{\delta}}\left(T_{{\delta}}\left(U_{0}, s\right)\right)-I_{{\delta}}\left(U_{0}\right)\right\|_{X^0_{{\delta}_{0}}} \\
\leq & \frac{\varepsilon}{6}+\left\|I_{{\delta}}\left(T_{{\delta}}\left(U_{0}, s\right)\right)-T_{{\delta}_{0}}\left(I_{{\delta}}\left(U_{0}\right), s\right)\right\|_{X^0_{{\delta}_{0}}} \\
& +\left\|T_{{\delta}_{0}}\left(I_{\delta}\left(U_{0}\right), s\right)-I_{{\delta}}\left(U_{0}\right)\right\|_{X^0_{{\delta}_{0}}} \leq \frac{\varepsilon}{2}.
\end{align*}

Finally, let $\gamma =\min\{ \gamma_1 , \gamma_2\}$. For any $h \in \mathcal{U}$ with  $d_{C^2}(h,1_{\Omega})=\delta$, $U_{0}$, $\tilde{U}_{0} \in A_{{\delta}}$ and $t, s \in[0,1]$, we suppose $\|U_{0}-\tilde{U}_{0}\|_{X^0_{\delta}}+|t-s|<\gamma$. Then, by Steps 1 and 2, we have
\begin{align*}
\|T_{{\delta}}\left(U_{0}, t\right)-T_{{\delta}}(\tilde{U}_{0}, s)\|_{X^0_{{\delta}}} \leq & \|T_{{\delta}}\left(U_{0}, t\right)-T_{{\delta}}(\tilde{U}_{0}, t)\|_{X^0_{{\delta}}} +\|T_{{\delta}}(\tilde{U}_{0}, t)-T_{{\delta}}(\tilde{U}_{0}, s)\|_{X^0_{{\delta}}} \\
\leq & \frac{\varepsilon}{2}+\|T_{{\delta}}\left(T_{{\delta}}\left(U_{0}, s\right), t-s\right)-T_{{\delta}}(\tilde{U}_{0}, s)\|_{X^0_{{\delta}}} \leq \varepsilon,
\end{align*}
and so completes the proof of Lemma \ref{continuity on h}.
\end{proof}
\noindent
\begin{proof}[Proof of Theorem $\ref{thm2}.$]
It is enough to show that there exists a residual subset $\mathcal R$ of $P(\Omega)$ such that the map $\Phi:\mathcal R \to  \mathcal{DS}$ given by $\Phi(h) = T_\delta$ is continuous, where $\delta=d_{C^2}(h, 1_{\Omega})$, and $T_\delta$ is the dynamical system on its global attractor $\mathcal {A}_\delta$ induced by the equation \eqref{perturbed}.

To prove this, for any $h, k \in \mathcal{P}(\Omega)$, we let $ d_{C^2}(h,1_{\Omega})=\delta$ and $d_{C^2}(k,1_{\Omega})=\delta'$. Define an isomorphism $i_{hk}: L^2(\Omega_{\delta}) \rightarrow L^2(\Omega_{\delta^\prime})$ by
$$
i_{hk}(u) = u \circ (h \circ k^{-1}),\quad \forall  u \in L^2(\Omega_\delta).
$$
Then we observe that the restriction of $i_{hk}$ on $H^1_0(\Omega_\delta)$ is also an isomorphism to $H^1_0(\Omega_{\delta^\prime})$.

Let $I_{hk}: X^0_{\delta}\rightarrow X^0_{\delta^\prime}$ be the map defined by
$$
I_{hk}(U) =(i_{hk}(u), i_{hk}(v)),\quad \forall U=(u, v) \in X^0_{\delta}.
$$
Then $I_{hk}$ is an isomorphism, and we see that for any $\varepsilon>0$, there is $\gamma>0$ such that if $d_{C^2} (h,k)< \gamma$, then $I_{hk}$ and $I_{kh}$ are $\varepsilon$-immersions.

We first claim that there is a residual subset $\mathcal R$ such that for any $k \in \mathcal R$ with $d_{C^2}(k,1_{\Omega}) = \delta^\prime$ and $\varepsilon>0$, there is a neighborhood $\mathcal W$ of $k$ such that for any $h \in W$ with $d_{C^2}(h,1_{\Omega}) = \delta$, the $\varepsilon$-immersion $I_{hk}$ satisfies
$$
I_{hk}(\mathcal A_\delta) \subset B(\mathcal A_{\delta^\prime}, \varepsilon)\quad \mathcal A_{\delta^{\prime}} \subset B(I_{hk}(\mathcal A_{\delta})),
$$
where $B(\mathcal A_{\delta^{\prime}}, \varepsilon) = \{U \in X_{\delta^\prime}^0: \| U-V\|_{X_{\delta^{\prime}}^0} < \varepsilon ~ \mbox{for~some}~V \in  \mathcal A_{\delta^\prime}\}$.
This proof can be carried out in a similar way to Step 1 of the proof of Theorem 1.2 in \cite{LNT} .

Next if we apply Lemma \ref{continuity on h}, we can prove that the map $\Phi: \mathcal{P}(\Omega) \to \mathcal{DS}$ given by $\Phi(h) =T_\delta$ is continuous on $\mathcal{R}$, whose  proof is similar to Step 2 of the proof of Theorem 1.2 in \cite{LNT}.
\end{proof}

\vskip 5mm

\noindent{\bf Acknowledgements.}
This paper is supported by the NRF grant funded by the Ministry of Science and ICT (RS-2023-NR076395) and the NRF grant funded by the Ministry of Education (RS-2024-00442775).

\end{document}